\documentclass[12pt, reqno]{amsart}
\usepackage{amsmath, amsthm, amscd, amsfonts, amssymb, graphicx, color}
\usepackage[bookmarksnumbered, colorlinks, plainpages]{hyperref}
\input{mathrsfs.sty}
\hypersetup{colorlinks=true,linkcolor=red, anchorcolor=green, citecolor=cyan, urlcolor=red, filecolor=magenta, pdftoolbar=true}
\newtheorem{theorem}{Theorem}[section]
\newtheorem{lemma}[theorem]{Lemma}
\newtheorem{proposition}[theorem]{Proposition}
\newtheorem{corollary}[theorem]{Corollary}
\usepackage{enumerate}
\theoremstyle{definition}
\newtheorem{definition}[theorem]{Definition}
\newtheorem{example}[theorem]{Example}

\newtheorem{Theorem}{\quad Theorem}[section]

\newtheorem{remark}[Theorem]{\quad Remark}
\numberwithin{equation}{section}
 
\begin{document}
\title[ Approximate numerical radius orthogonality] {Approximate numerical radius orthogonality}

\author[M. Amyari and M. Moradian Khibary]
{Maryam Amyari $^{1*}$ \MakeLowercase{and} Marzieh Moradian Khibary$^2$}

\address{$^{1*}$ Department of Mathematics, Mashhad Branch,
Islamic Azad University, Mashhad, Iran}
\email{maryam\_amyari@yahoo.com and amyari@mshdiau.ac.ir}

\address{$^2$ Department of Mathematics, Farhangian University, Mashhad, Iran}
\email{mmkh926@gmail.com}


\keywords{Hilbert space, numerical radius, approximate numerical radius orthogonality, Numerical radius derivation.\\
$*$Corresponding author}
%%%%%%%%%%%%%%%%%%%%%%%%%%%%%%%%%%%%%%%%%
\maketitle

\begin{abstract}
We introduce the notion of approximate numerical radius (Birkhoff) orthogonality and investigate its significant properties. Let $T, S\in \mathbb{B}(\mathscr{H})$ and  $\varepsilon \in [0, 1)$. We say that $T$ is approximate numerical radius  orthogonal to $S$ and we write $T\perp^{\varepsilon}_{\omega} S$ if
$$\omega^2(T+\lambda S)\geq \omega^2(T)-2\varepsilon \omega(T) \omega(\lambda S)\,\,\, \text{for all }\lambda\in\mathbb{C}.$$
We show that $T\perp^{\varepsilon}_{\omega} S$ if and only if
$\displaystyle\inf_{\theta\in [0, 2\pi)} D^{\theta}_{\omega}(T, S)
\geq -\varepsilon \omega(T) \omega(S)$ in which $D^{\theta}_{\omega}(T, S)=\displaystyle\lim_{r\to 0^+} \frac{\omega^2(T+re^{i\theta} S)-\omega^2(T)}{2r}$; and this occurs  if and only if
for every $\theta\in[0,2\pi)$, there exists a sequence $\{x_n^{\theta}\}$  of unit vectors in $\mathscr{H}$ such that
$$\displaystyle\lim_{n\to \infty} |\langle Tx^{\theta}_n, x^{\theta}_n\rangle|=\omega(T),\,\, \text{and}\,\, \displaystyle\lim_{n\to \infty} {\rm  Re}\{e^{-i\theta} \langle Tx^{\theta}_n, x^{\theta}_n\rangle\overline{\langle Sx^{\theta}_n, x^{\theta}_n\rangle}\}\geq -\varepsilon \omega(T) \omega(S),$$ where  $\omega(T)$ is the numerical radius  of $T$.
\end{abstract} \maketitle
%%%%%%%%%%%%%%%%%%%%%%%%%%%%%%%%%%%
\section{Introduction and preliminaries}

One of the most important types of orthogonality in the setting of normed spaces is the Birkhoff--James orthogonality. Let $(X, \|\cdot\|)$ be a linear normed space and $x, y \in X$. A vector $x$ is called Birkhoff--James orthogonal to a vector $y$, written as $x \perp_B y$,
if $\|x+\lambda y\|\geq \|x\|$ for every $\lambda \in \mathbb{C}$.

Many mathematicians generalized the notion of Birkhoff--James orthogonality in the set up of normed spaces. Dragomir \cite{SDR} introduced the notion of $\varepsilon$-Birkhoff--James orthogonality in a real normed space $X$ as follows.

Let $x, y \in X$ and  $\varepsilon \in [0, 1)$. We say that $x$ is  $\varepsilon$-Birkhoff--James orthogonal to $y$ if \[\|x+\lambda y\|\geq (1-\varepsilon)\|x\|\] for all $\lambda \in \mathbb{R}$.

Chmieli\'{n}ski \cite{JCH}  introduced  another notion of $\varepsilon$-Birkhoff--James orthogonalityin the setting of normed spaces, where a vector $x$ is said to be approximate Birkhoff--James orthogonal to a vector $y$, written as  $x \perp^{\varepsilon}_B y$, if \[\|x+\lambda y\|^2\geq \|x\|^2-2\varepsilon\|x\|\|\lambda y\|\] for all $\lambda \in \mathbb{R}$. He also proved that in an inner product space $x \perp^{\varepsilon}_B$ if and only if $|\langle x, y\rangle|\leq \varepsilon\|x\|\|y\|$. The notion  of approximate orthogonality has been developed in several
settings; see \cite{CSW, AK, MOS}.

STANDING NOTATION: Let  $(\mathscr{H},\langle \cdot ,\cdot \rangle)$ be a Hilbert space and
 $\mathbb{B}(\mathscr{H})$ be the algebra of all bounded linear operators on $\mathscr{H}$ with the identity $I$. A capital letter denotes a bounded linear operator in ${\mathbb B}({\mathscr H})$. The numerical radius of $T$ is defined by
$$\omega(T) = \sup\{|\langle Tx, x\rangle|: x\in\mathscr{H}, \|x\| = 1\}.$$ Throughout the paper, we assume  $\varepsilon \in [0, 1)$ and for calculate the numerical raduis, we use some formulas appeared in \cite[Theorem 3]{RS}, \cite[Theorem 2.1]{PS}, and \cite[Theorem 3.7]{HK}.

Recently, the authors of \cite{AM} introduced the notion of numerical radius (Birkhoff) orthogonality for operators in $\mathbb{B}(\mathscr{H})$. They called $T$ the numerical radius orthogonal to $S$, denoted by $T\perp_{\omega B}S$ if $$\omega(T+\lambda S)\geq\omega(T)\quad\mbox{for all}\quad\lambda\in\mathbb{C}.$$
We introduce an approximate version of the above notion and present some of its characterizations. The paper is organized as follows.

In section 2, we introduce the notion of approximate  numerical radius orthogonality ``$\perp^{\varepsilon}_{\omega} $'' and prove that for two operators $T$ and $S$, it holds that $T\perp^{\varepsilon}_{\omega}S$ if and only if
for every $\theta\in[0,2\pi)$ there exists a sequence $\{x_n^{\theta}\}$  of unit vectors in $\mathscr{H}$ such that
$$\displaystyle\lim_{n\to \infty} |\langle Tx^{\theta}_n, x^{\theta}_n\rangle|=\omega(T),\,\, \text{and}\,\, \displaystyle\lim_{n\to \infty} {\rm  Re}\{e^{-i\theta} \langle Tx^{\theta}_n, x^{\theta}_n\rangle\overline{\langle Sx^{\theta}_n, x^{\theta}_n\rangle}\}\geq -\varepsilon \omega(T) \omega(S).$$

In section 3, we introduce the notion of numerical radius derivation ($\omega$-
derivation) and show that $T\perp^{\varepsilon}_{\omega}S$ if and only if
$\displaystyle\inf_{\theta\in [0, 2\pi)} D^{\theta }_{\omega}(T, S)
\geq -\varepsilon \omega(T) \omega(S)$, where $D^{\theta}_{\omega}(T, S)=\displaystyle\lim_{r\to 0^+} \frac{\omega^2(T+re^{i\theta} S)-\omega^2(T)}{2r}$.

%%%%%%%%%%%%%%%%%%%%%%%%%%%%%%%
%%%%%%%%%%%%%%%%%%%%%%%%%%%%%%%

\section{Approximate orthogonality}

In this section, we introduce the notion of approximate numerical radius orthogonality and state some of its basic properties.
\begin{definition}
We say that an operator $T$ is approximate numerical radius (Birkhoff) orthogonal to an operator $S$ and we write $T\perp^{\varepsilon}_{\omega} S$ if
$$\omega^2(T+\lambda S)\geq \omega^2(T)-2\varepsilon \omega(T) \omega(\lambda S) \,\, \text{for all }\lambda\in\mathbb{C}.$$
\end{definition}
It is easy to see that  $T\perp^{\varepsilon}_{\omega} S$ and $\alpha T\perp^{\varepsilon}_{\omega} \beta S \,\,(\alpha, \beta \in \mathbb{C})$ are equivalent.
 The following example shows that the relation $\perp^{\varepsilon}_{\omega}$ is not symmetric, in general.
\begin{example}
Suppose that
$ T=
\begin{bmatrix}
i&0\\
0&0
\end{bmatrix}$
and
$S=
\begin{bmatrix}
0&1\\
0&-1
\end{bmatrix}$
are in $\mathbb{M}_2(\mathbb{C})$.
Simple computations show that
$\omega(T)=1$ $\omega(S)=\dfrac{1+\sqrt{2}}{2}$. Further,
$$\omega(T+\lambda S)=\omega\left( \begin{bmatrix}
i&\lambda\\
0&-\lambda
\end{bmatrix}\right)\geq \max\{|\lambda|, 1, \frac{|\lambda|}{2}\}.$$
Hence, $\omega^2(T+\lambda S)\geq \omega^2(T)-2\varepsilon \omega(T) \omega( \lambda S )$, that is $T\perp^ \varepsilon_{\omega} S$.

For $\lambda=\frac{-i}{2}$, we get
$$\omega(S-\frac{i}{2} T)=\omega\left( \begin{bmatrix}
1&1\\
0&-1
\end{bmatrix}\right)=\frac{1+\sqrt{13}}{4}\approx 1.151$$
$\omega^2(S-\frac{i}{2}T)\approx 1.325$ and $\omega^2(S)\approx 1.457$. Thus,
$\omega^2(S-\frac{i}{2}T)\leq \omega^2(S)-2\varepsilon \frac{1}{2}\omega (S) \omega (T)$  for $\varepsilon \in (0, 0.01)$.
Hence, $S\not\perp^\varepsilon_{\omega} T$.
\end{example}
%%%%%%%%%%%%%%%%%%%%%%%%
The following proposition yields some relations between approximate Birkhoff--James orthogonality  $\perp^{\varepsilon}_{ B}$ and approximate numerical radius orthogonality $\perp^{\varepsilon}_{\omega}$ under some conditions. The proof is easy and so we omit it.
\begin{proposition}
(i) If $T=T^*$, then  $T\perp^{\varepsilon}_{\omega} S$ implies that  $T\perp^{\varepsilon}_{ B} S$.

(ii) If $T^2=0$, then $T\perp^{\varepsilon}_{ B} S$ entails that  $T\perp^{\varepsilon}_{\omega} S$.
\end{proposition}
%%%%%%%%%%%%%%%%%%%%%%%%%%%%%%%%
The next example shows that $T\perp^{\varepsilon}_{ B} S$ does not entail
$T\perp^{\varepsilon}_{\omega} S$, in general.
\begin{example}
Suppose that
$
T=\begin{bmatrix}
0&1\\
0&-1
\end{bmatrix}
$
and
$
S=\begin{bmatrix}
1&0\\
0&0
\end{bmatrix}
$
 are in $\mathbb{M}_2(\mathbb{C})$.
 Then $\|T\|=
 \sqrt{2}$ and $\|S\|=1$ and for every $\lambda \in \mathbb{C}$,  we have
 $\| T+\lambda S \|^2= \dfrac{2+|\lambda|^2+\sqrt{4+|\lambda|^4}}{2}$. Hence
$\|T+ \lambda S\|^2\geq 2\geq 2-2\varepsilon|\lambda|=2-\sqrt{2}\varepsilon \|T\|\|\lambda S\|$. Thus, $T\perp^{\varepsilon}_{ B} S$.

Also we have
$\omega( T)=\frac{1+\sqrt{2}}{2},\,\, \omega (S)=1$, and
$$\omega(T+\lambda S)=\omega\left( \begin{bmatrix}
\lambda &1\\
0&-1
\end{bmatrix}\right)=\frac{\sqrt{5}}{2}.$$
For $\lambda=1$,  we reach $1.24\approx \omega^2(T+ S)< \omega^2(T)-2\varepsilon \omega(T) \omega(S)\approx  1.43$ showing that $T\not\perp^{\varepsilon}_{\omega} S$ for $\varepsilon \in (0, 0.01)$.
\end{example}
%%%%%%%%%%%%%%%%%%%%%%%%%%%%%%%%%%%%%%%%%%%%%%%%%%%
%%%%%%%%%%%%%%%%%%%%%%%%%%%%%%%%%%%%%%%%%%%%%%%%%%%
%%%%%%%%%%%%%%%%%%%%%
We give an example of two operators $T$ and $S$ such that  $T\not\perp_{\omega B} S$  while $T \perp ^\varepsilon _{\omega} S$.
\begin{example}
Suppose that
$T=
\begin{bmatrix}
2&0\\
0&0
\end{bmatrix}
\,\, and \,\,
S=
\begin{bmatrix}
1&1\\
0&1
\end{bmatrix}
$
are in $\mathbb{M}_2(\mathbb{C})$. Straightforward computations give us
$\omega (T)=2,\,\, \omega (S)=\frac{3}{2}$.
If $\lambda=-1$, then
$$\omega(T-S)=\omega\left( \begin{bmatrix}
1 &-1\\
0& -1
\end{bmatrix}\right)=\frac{\sqrt{5}}{2}< 2.$$
Hence, $T\not\perp_{\omega B} S$. We also have
$$\omega(T+\lambda S)=\omega\left( \begin{bmatrix}
2+\lambda & \lambda\\
0& \lambda
\end{bmatrix}\right)\geq \max \{|\lambda|,|\lambda +2|\}.$$
Thus,
$\omega^2(T+\lambda S)\geq \max \{|\lambda|^2, |2+\lambda|^2\}\geq
4-6\varepsilon |\lambda |=\omega^2(T)-2\varepsilon |\lambda | \omega(T) \omega(S)$.
Therefore, $T \perp ^\varepsilon _{\omega} S$ for each $\varepsilon \in (\frac{2}{3}, 1)$.
\end{example}
%%%%%%%%%%%%%%%%%%%%%%%%%%%%%%%
Mal et al. \cite[Theorem 2.3]{AM} characterized  the numerical radius Birkhoff orthogonality of operators on a  Hilbert space. In \cite{MMM}, the authors investigated some aspects of numerical radius orthogonality.
Inspired by these papers, we characterize the approximate numerical radius orthogonality of operators acting on a Hilbert space.

\begin{theorem}\label{*}
The relation $T\perp^{\varepsilon}_{\omega} S$ holds if and only if for every $\theta \in [0, 2\pi)$, there exists a sequence $\{x^{\theta}_n\}_{n\in\mathbb{N}}$ of unit vectors in $\mathscr{H}$ such that the
following two conditions hold:

(i) $\displaystyle\lim_{n\to \infty} |\langle Tx^{\theta}_n, x^{\theta}_n\rangle| =\omega(T)$

(ii) $\displaystyle\lim_{n\to \infty} {\rm  Re}\{e^{-i\theta} \langle Tx^{\theta}_n, x^{\theta}_n\rangle\overline{\langle Sx^{\theta}_n, x^{\theta}_n\rangle}\}\geq -\varepsilon \omega(T)\omega(S)$.
\end{theorem}
\begin{proof}
($\Longleftarrow$) Let $\lambda\in\mathbb{C}$. Then $\lambda= |\lambda|e^{i\theta}$ for some $\theta \in [0, 2\pi)$. By the assumption, there exists a sequence $\{x^{\theta}_n\}_{n\in\mathbb{N}}$ of unit vectors in $\mathscr{H}$ such that (i) and (ii) hold. Thus,
\begin{align*}
\omega^2(T+\lambda S)& \geq \lim_{n\to \infty}|\langle (T+\lambda S)x^{\theta}_n, x^{\theta}_n\rangle|^2\\
&=\lim_{n\to \infty}\left(|\langle Tx^{\theta}_n, x^{\theta}_n \rangle|^2+2|  \lambda|{\rm Re}\{e^{-i\theta}\langle  Tx^{\theta}_n, x^{\theta}_n\rangle\overline{\langle Sx^{\theta}_n, x^{\theta}_n\rangle}\}+|\lambda|^2|\langle Sx^{\theta}_n, x^{\theta}_n\rangle|^2 \right)\\
&\geq \lim_{n\to \infty}\left(|\langle Tx^{\theta}_n, x^{\theta}_n \rangle|^2+2|  \lambda|{\rm Re}\{e^{-i\theta}\langle  Tx^{\theta}_n, x^{\theta}_n\rangle\overline{\langle Sx^{\theta}_n, x^{\theta}_n\rangle}\}\right)\\
&\geq \omega^2(T)-2\varepsilon \omega(T) \omega(\lambda S)
\end{align*}
Thus, $T\perp^{\varepsilon}_{\omega} S$.

($\Longrightarrow$)
Let $\theta\in[0,2\pi)$. We drive from $T\perp^{\varepsilon}_{\omega} S$ that
$\omega^2(T+\lambda  S)\geq\omega^2(T)-2\varepsilon \omega(T) \omega(\lambda S)$ for all $\lambda\in\mathbb{C}$. Hence, $\omega^2\left(T+\frac{ e^{i\theta}}{n} S\right)\geq\omega^2(T)-2\varepsilon \omega(T)\omega\left(\frac{e^{i\theta}}{n} S\right)$ for all $n\in \mathbb{N}$.

For every $n\in \mathbb{N}$ there exists $x^{\theta}_n$ with $\|x^{\theta}_n\|=1$
such that
$$\omega^2\left(T+\frac{e^i{\theta}}{n} S\right)-\frac{1}{n^2}< \left|\left\langle \left(T+\frac{ e^{i\theta}}{n} S\right)x^{\theta}_n, x^{\theta}_n\right\rangle\right|^2,$$
whence
\begin{align}\label{amy1}
&\hspace{-1cm} \omega^2(T)-\frac{2\varepsilon}{n} \omega(T)\omega( S)-\frac{1}{n^2}\nonumber \\ 
&\leq\omega^2\left(T+\frac{e^{i\theta}}{n} S\right )-\frac{1}{n^2} \nonumber \\
& <\left|\left\langle \left(T+\frac{ e^{i\theta}}{n} S\right)x^{\theta}_n, x^{\theta}_n\right\rangle\right|^2\nonumber\\
&=\Big(|\langle Tx^{\theta}_n, x^{\theta}_n \rangle|^2+\frac{2}{n} {\rm Re}\{e^{-i\theta}\langle  Tx^{\theta}_n, x^{\theta}_n\rangle\overline{\langle Sx^{\theta}_n, x^{\theta}_n\rangle}\}+\frac{1}{n^2}|\langle Sx^{\theta}_n, x^{\theta}_n\rangle|^2 \Big).
\end{align}
Therefore,
\begin{align*}
\frac{n}{2}(\omega^2(T)-|\langle Tx^{\theta}_n, x^{\theta}_n \rangle|^2)
&<{\rm Re}\{e^{-i\theta}\langle  Tx^{\theta}_n, x^{\theta}_n\rangle\overline{\langle Sx^{\theta}_n, x^{\theta}_n\rangle}\}\\
&+\frac{1}{2n}\omega^2(S)+\frac{2}{n}+\varepsilon \omega(T)\omega( S)\qquad (n\in \mathbb{N}),
\end{align*}
and hence
\begin{align}\label{21}
0\leq {\rm Re}\{e^{-i\theta}\langle  Tx^{\theta}_n, x^{\theta}_n\rangle\overline{\langle Sx^{\theta}_n, x^{\theta}_n\rangle}\}
+\frac{1}{2n}\omega^2(S)+\frac{2}{n}+\varepsilon \omega(T)\omega( S) \qquad (n\in \mathbb{N}).
\end{align}
Note that \{$\langle Tx^{\theta}_n, x^{\theta}_n \rangle\}$ and $\{\langle Sx^{\theta}_n, x^{\theta}_n \rangle\}$ are two bounded sequence in $\mathbb{C}$. Therefore, by  passing to subsequences of $\{x^{\theta}_n\}_{n\in\mathbb{N}}$, if necessary, we can assume that these twosequences are convergent.  Now, inequality \eqref{21} implies that
$$\lim_{n\to \infty}{\rm Re}\{e^{-i\theta}\langle  Tx^{\theta}_n, x^{\theta}_n\rangle\overline{\langle Sx^{\theta}_n, x^{\theta}_n\rangle}\}\geq -\varepsilon \omega(T)\omega(S).$$
Thus, (ii) is valid.

We shall prove (i). It follows from \eqref{amy1} that
\begin{align*}
\omega^2(T)&-\frac{2\varepsilon}{n} \omega(T)\omega(S)-\frac{1}{n^2}\\
&\leq \Big(|\langle Tx^{\theta}_n, x^{\theta}_n \rangle|^2+\frac{2}{n} {\rm Re}\{e^{-i\theta}\langle  Tx^{\theta}_n, x^{\theta}_n\rangle\overline{\langle Sx^{\theta}_n, x^{\theta}_n\rangle}\}+\frac{1}{n^2}|\langle Sx^{\theta}_n, x^{\theta}_n\rangle|^2 \Big) \\
&\leq |\langle Tx^{\theta}_n, x^{\theta}_n \rangle|^2+\frac{2}{n} |\langle  Tx^{\theta}_n, x^{\theta}_n\rangle||\langle Sx^{\theta}_n, x^{\theta}_n\rangle|+\frac{1}{n^2}|\langle Sx^{\theta}_n, x^{\theta}_n\rangle|^2\\
&\leq |\langle Tx^{\theta}_n, x^{\theta}_n \rangle|^2+\frac{2}{n} \omega(T)\omega( S)+\frac{1}{n^2}\omega^2 (S),
\end{align*}
for all $n\in \mathbb{N}$. Hence
\begin{align*}
\omega^2(T)\geq |\langle Tx^{\theta}_n, x^{\theta}_n \rangle|^2 &\geq
\omega^2(T)-\frac{2\varepsilon}{n} \omega(T)\omega(S)-\frac{1}{n^2}-\frac{2}{n} \omega(T)\omega(S)-\frac{1}{n^2}\omega^2 (S),
\end{align*}
for all $n\in \mathbb{N}$.
Therefore,$\displaystyle\lim_{n\to \infty}|\langle Tx^{\theta}_n, x^{\theta}_n \rangle|= \omega(T)$.
\end{proof}
%%%%%%%%%%%%%%%%%%%%%%%%%%%%%%%%%
\begin{remark}
Due to the homogeneity of relations $\perp^{\varepsilon}_{\omega} $, without loss of generality, we may assume that $\omega(T)= \omega(S)=1$. Then $T\perp^{\varepsilon}_{\omega} S$
if and only if for every $\theta \in [0, 2\pi)$, there exists a sequence $\{x^{\theta}_n\}_{n\in\mathbb{N}}$ of unit vectors in $\mathscr{H}$ such that the
following two conditions hold:

(i) $\displaystyle\lim_{n\to \infty} |\langle Tx^{\theta}_n, x^{\theta}_n\rangle|=1$,

(ii) $\displaystyle\lim_{n\to \infty} {\rm  Re}\{e^{-i\theta} \langle Tx^{\theta}_n, x^{\theta}_n\rangle\overline{\langle Sx^{\theta}_n, x^{\theta}_n\rangle}\}\geq -\varepsilon$.
\end{remark}
%%%%%%%%%%%%%%%%%%%%%%%%
As an application of Theorem \ref{*}, we get the following results. 
 \begin{corollary}\label{22}
If $T\perp^{\varepsilon}_{\omega} I$, then $I \perp^{\varepsilon}_{\omega}T$.
 \begin{proof}
Let $\theta \in [0, 2\pi)$. Since $T\perp^{\varepsilon}_{\omega}I$, for $\varphi=2\pi -\theta $  there exists a sequence $\{x^{\varphi}_n\}_{n\in\mathbb{N}}$ of unit vectors in $\mathscr{H}$ such that
$\displaystyle\lim_{n\to \infty} |\langle Tx^{\varphi}_n, x^{\varphi}_n\rangle| =\omega(T)$ and\\
 $\displaystyle\lim_{n\to \infty}{\rm  Re}\{e^{-i\varphi}\langle Tx^{\varphi}_n, x^{\varphi}_n\rangle  \overline{\langle Ix^{\varphi}_n, x^{\varphi}_n\rangle}\}\geq -\varepsilon \omega (T)$.
 Put $x^{\theta}_n=x^{\varphi}_n$. Then $\displaystyle\lim_{n\to \infty} |\langle Ix^{\theta}_n, x^{\theta}_n\rangle|=1$ and
 \begin{align*}
\lim_{n\to \infty}{\rm  Re}\{e^{-i\theta}\langle Ix^{\theta}_n, x^{\theta}_n\rangle  \overline{\langle Tx^{\theta}_n, x^{\theta}_n\rangle}\}&=
 \displaystyle\lim_{n\to \infty}{\rm  Re}\{\overline{e^{-i\theta}\langle Ix^{\theta}_n, x^{\theta}_n\rangle}\langle Tx^{\theta}_n, x^{\theta}_n\rangle\}\\
 &=\lim_{n\to \infty}{\rm  Re}\{e^{-i\varphi}\langle Tx^{\varphi}_n, x^{\varphi}_n\rangle  \overline{\langle Ix^{\varphi}_n, x^{\varphi}_n\rangle}\}\geq  -\varepsilon \omega (T),
\end{align*}
whence  $I\perp^{\varepsilon}_{\omega}T$.
\end{proof}
\end{corollary}
%%%%%%%%%%%%%%%%%%%%%%%%%%%%

%%%%%%%%%%%%%%%%%%%%%%%%%%%%%%%%%%%
Given an operator $T$, the set of all sequences in the closed unit ball $\mathscr{H}$ at which $T$ attains its numerical radius in limits is denoted by
$$M^*_{\omega(T)}=\{\{x_n\}:\,\,\,\|x_n\|=1, \lim_{n\to \infty}|\langle Tx_n, x_n\rangle|=\omega(T)\}.$$

 In the following result, we show that under some mild conditions, $\perp^{\varepsilon}_{\omega}$ behaves like a symmetric relation. Recall that the Crawford number of an operator $T$ is defined by
 $$c(T)=\inf \{|\langle Tx,x\rangle|:\,\|x\|=1\}.$$
\begin{proposition}
Let $c(T)\neq 0$.
 If $T\perp^{\varepsilon}_{\omega} S$ and  $M^*_{\omega(S)}\cap M^*_{\omega(T+\lambda S)}\neq \emptyset$ for all $\lambda\in\mathbb{C}$, then $S\perp^{\varepsilon}_{\omega} T$.
\end{proposition}
\begin{proof}
Let $\lambda\in\mathbb{C}$. Put $\beta:=\dfrac{\omega(S)}{c(T) }$. Since $\perp^{\varepsilon}_{\omega}$ is homogeneous, we have
$\beta T\perp^{\varepsilon}_{\omega} S$. Hence
$\omega^2(\beta T+\overline{\lambda} S)\geqslant \omega^2(\beta T)-2 \varepsilon \omega(\beta T) \omega( \overline{\lambda} S)$. Let $\{x_n\}\in M^*_{\omega(S)}\cap M^*_{\omega(T+\frac{\overline{\lambda}}{\beta} S)}$. We have
\begin{align*}
 \omega^2(\beta T)&-2 \varepsilon \omega(\beta T) \omega( \overline{\lambda} S)\\
 &\leq\omega^2(\beta T+\overline{\lambda} S)\\
& =\lim_{n\to\infty}|\langle\beta T+\overline{\lambda} Sx_n,x_n\rangle|^2\\
&=\lim_{n\to\infty}\Big(|\langle\beta Tx_n,x_n\rangle|^2+2{\rm Re}  \lambda\langle \beta Tx_n,x_n\rangle\overline{\langle Sx_n,x_n\rangle}+|\lambda|^2|\langle Sx_n,x_n\rangle|^2 \Big).
\end{align*}
From $\lim_{n\to\infty}|\langle\beta Tx_n,x_n\rangle|^2 \leq \omega^2(\beta T)$,
we infer that 
\begin{align}\label{amy5}
\lim_{n\to\infty}\Big(2{\rm Re}  \lambda\langle \beta Tx_n,x_n\rangle\overline{\langle Sx_n,x_n\rangle}+|\lambda|^2|\langle Sx_n,x_n\rangle|^2 \Big)\geq -2 \varepsilon \omega(\beta T) \omega( \overline{\lambda} S).
\end{align}
From $\beta=\dfrac{\omega(S)}{c(T) }$, we conclude that $\lim_{n\to\infty}(\beta^2|\langle  Tx_n,x_n\rangle|^2-|\langle Sx_n,x_n\rangle|^2)\geq 0$.
Thus
\begin{align*}
\omega^2(S+\lambda \beta T) &\geq \lim_{n\to\infty}|\langle ((S+\lambda \beta  T)x_n,x_n\rangle|^2\\
&= \lim_{n\to\infty} \Big(|\langle Sx_n,x_n\rangle|^2+2{\rm Re}\lambda\langle \beta Tx_n,x_n\rangle\overline{\langle Sx_n,x_n\rangle}+
|\lambda|^2|\langle\beta Tx_n, x_n\rangle|^2\Big)\\
&\geq \lim_{n\to\infty} \Big(|\langle Sx_n,x_n\rangle|^2+2{\rm Re} \lambda\langle \beta Tx_n,x_n\rangle\overline{\langle Sx_n,x_n\rangle}+|\lambda|^2|\langle Sx_n,x_n\rangle|^2\Big)\\
&\geq \lim_{n\to\infty} |\langle Sx_n,x_n\rangle |^2-2 \varepsilon \omega(\beta T) \omega(\overline{\lambda} S) \hspace{3cm}({\rm by~} \eqref{amy5}).
\end{align*}
It follows from the assumption that $\lim_{n\to\infty} |\langle Sx_n,x_n\rangle|=\omega(S)$. Therefore
 $$\omega^2(S+\lambda \beta T)\geq \omega^2(S)-2 \varepsilon \omega(S) \omega( \beta\lambda T),$$
since $\omega(\mu S)=|\mu|\omega(S)$ for each $\mu \in\mathbb{C}$. Thus $S\perp^{\varepsilon}_{\omega} T$.
\end{proof}
%%%%%%%%%%%%%%%%%%%%%%%%%%
For compact operators, in particular in the case where $\mathscr{H}$ is finite dimensional, Theorem \ref{*} yields the following result.

\begin{theorem}\label{**}
Let $T$ and $S$ be two compact operators. Then $T\perp^{\varepsilon}_{\omega} S$ holds
if and only if for every $\theta \in [0, 2\pi)$, there exists a unit vector $x^{\theta}\in \mathscr{H}$  such that
$|\langle Tx^{\theta}, x^{\theta}\rangle|=\omega(T)$ and
$ {\rm  Re}\{e^{-i\theta} \langle Tx^{\theta}, x^{\theta}\rangle\overline{\langle Sx^{\theta}, x^{\theta}\rangle}\} \geq -\varepsilon \omega(T) \omega(S)$.
\end{theorem}
\begin{proof}
($\Longleftarrow$) It is obvious by Theorem \ref{*}.

($\Longrightarrow$)  Let $\theta \in [0, 2\pi)$. It follows from Theorem \ref{*} that there exists a sequence $\{x^{\theta}_n\}_{n\in\mathbb{N}}$ of unit vectors in $\mathscr{H}$ such that both (i) $\displaystyle\lim_{n\to \infty} |\langle Tx^{\theta}_n, x^{\theta}_n\rangle|=\omega(T)$ and (ii) $\displaystyle\lim_{n\to \infty} {\rm  Re}\{e^{-i\theta} \langle Tx^{\theta}_n, x^{\theta}_n\rangle\overline{\langle Sx^{\theta}_n, x^{\theta}_n\rangle}\}\geq -\varepsilon \omega(T) \omega(S)$ hold.

Since the closed unit ball of $\mathscr{H}$ is weakly compact, $\{x^{\theta}_n\}$ has a weakly convergent subsequence. Without loss of generality, we assume that $\{x^{\theta}_n\}$ is weakly converges, say to $x^{\theta}$. H ence,  $\langle x^{\theta}_n-x^{\theta},T^*y\rangle \to 0$ as $n \to \infty$ for all $y\in \mathscr{H}$. Therefore $\{Tx^{\theta}_n\}$ weakly converges to  $Tx^{\theta}$. Similarly $\{Sx^{\theta}_n\}$ weakly converges to  $Sx^{\theta}$. 

On the other hand, since $\{x^{\theta}_n\}$ is norm-bounded and the operators $T$ and $S$ are compact, by passing to subsequences, we can assume that $\{Tx^{\theta}_n\}$ and $\{Sx^{\theta}_n\}$ are norm-convergent. Thus, $\lim_{n\to \infty}  Tx^{\theta}_n=Tx^{\theta}$ and $\lim_{n\to \infty}  Sx^{\theta}_n=Sx^{\theta}$ in the norm topology.
Therefore, $\lim_{n\to \infty} \langle Tx^{\theta}_n, x^{\theta}_n\rangle=\langle Tx^{\theta}, x^{\theta}\rangle$ and $\lim_{n\to \infty} \langle Sx^{\theta}_n, x^{\theta}_n\rangle=\langle Sx^{\theta}, x^{\theta}\rangle$. Now by considering (i) and (ii), the proof is completed.
\end{proof}
%%%%%%%%%%%%%%%%%%%%%%%%%%%%%%%%%%%
\begin{example}
Suppose that $x, y \in\mathscr{H}$ are unit vectors and $x\otimes y$ denotes the rank one operator defined by $(x\otimes y)(z) : = \langle z, y\rangle x,\,\, z\in\mathscr{H}$.
It is easy to see that $\omega(x\otimes y) = \frac{1}{2}\left(|\langle x, y\rangle| + \|x\otimes y\|\right)$ for all $x, y\in\mathscr{H}$. Hence, for the compact operator $x\otimes x$, we get $\omega(x\otimes x)=\|x\|^2$.
 From Theorem \ref{**}, $x\otimes x \perp^{\varepsilon}_{\omega} y\otimes y$
if and only if for every $\theta \in [0, 2\pi)$, there exists  a unit vector $x^{\theta}\in \mathscr{H}$ such that
$1=|\langle x^{\theta}, x\rangle|$ and
$ \cos\theta\,|\langle x^{\theta}, x\rangle|^2|\langle x^{\theta}, y\rangle|^2\geq -\varepsilon $. 

From the equality case in the Cauchy--Schwarz inequality and $1=|\langle x^{\theta}, x\rangle|$ we infer that $x^{\theta}=x$. If $x$ and $y$ are orthogonal, we arrive at $x\otimes x \perp^{\varepsilon}_{\omega} y\otimes y$.

Moreover, if $\varepsilon>0$ is given and take the unit vectors $x_\varepsilon, y_\varepsilon\in\mathscr{H}$ such that 
$\varepsilon < |\langle x_\varepsilon ,y_\varepsilon \rangle |^2$, and take $\theta_\varepsilon \in[0, \pi)$ such that $-1<\cos\theta_\varepsilon<\frac{-\varepsilon}{|\langle x_\varepsilon,y_\varepsilon\rangle|^2}$, then the inequality $\cos\theta_\varepsilon |\langle x_\varepsilon,y_\varepsilon\rangle|^2 <-\varepsilon$ ensures that $x\otimes x \not\perp^{\varepsilon}_{\omega} y\otimes y$.

\end{example}
%%%%%%%%%%%%%%%%%%%%%%%%%%%%%%%
\begin{proposition}
If  $T$ is a positive operator and $T\perp^{\varepsilon}_{\omega}S$, then $(T+I)\perp^{\varepsilon}_{\omega}S$.
\end{proposition}
\begin{proof}
Let $\theta \in [0, 2\pi)$. By the assumption, there exists a sequence $\{x^{\theta}_n\}$ of unit vectors such that $$\lim_{n\to \infty} |\langle Tx^{\theta}_n, x^{\theta}_n\rangle|=\omega(T),\,\,\text{and}\,\,\,\lim_{n\to \infty}{\rm Re}\{e^{-i\theta}\langle  Tx^{\theta}_n, x^{\theta}_n\rangle\overline{\langle Sx^{\theta}_n, x^{\theta}_n\rangle}\}
\geq -\varepsilon \omega(T) \omega(S).$$
 Since $T$ is positive,  $\omega(T+I)=\omega(T)+1$ and $ \displaystyle\lim_{n\to \infty} Re \langle Tx^{\theta}_n, x^{\theta}_n \rangle= \displaystyle\lim_{n\to \infty} \langle Tx^{\theta}_n, x^{\theta}_n \rangle$.\\
 Hence,
 \[\lim	_{n\to \infty}{\rm Re}\{e^{-i\theta}\overline{\langle Sx^{\theta}_n, x^{\theta}_n\rangle}\}
\geq -\varepsilon \omega(S)\]
 and
\begin{align*}
\lim_{n\to \infty}|\langle (T+I) x_n^{\theta}, x_n^{\theta}\rangle|^2&=\lim_{n\to \infty}\left(|\langle T x_n^{\theta},x_n^{\theta}\rangle|^2+|\langle I x_n^{\theta},x_n^{\theta}\rangle|^2+2{\rm  Re} \langle Tx^{\theta}_n, x^{\theta}_n\rangle\right)\\
&=\omega^2(T)+1+2\omega(T)=\omega^2(T+I).
\end{align*}
Thus,
\begin{align*}
&\hspace{-1cm}\lim_{n\to\infty} {\rm  Re} \{e^{-i\theta}\langle (T+I) x_n^{\theta},x_n^{\theta}\rangle\overline{\langle S x_n^{\theta},x_n^{\theta}\rangle}\}\\
&=\lim_{n\to\infty}{\rm  Re}\{e^{-i\theta}\langle Tx_n^{\theta},x_n^{\theta}\rangle \overline{\langle Sx^{\theta}_n, x^{\theta}_n\rangle}\}+\lim_{n\to\infty} {\rm  Re} \{e^{-i\theta}\langle I x_n^{\theta},x_n^{\theta}\rangle\overline{\langle S x_n^{\theta},x_n^{\theta}\rangle}\}\\
&\geq -\varepsilon \omega(T) \omega(S)-\varepsilon \omega(S)\\
&\geq -\varepsilon\omega(S)\omega(T+I).
\end{align*}
Therefore, $(T+I)\perp ^\varepsilon_{\omega} S$.\\
\end{proof}
%%%%%%%%%%%%%%%%%%%%%%%%%

Our last result of this section reads as follows.
\begin{proposition}
Let $S$ and $K$ be positive operators of norm one, $K\leq S$, and 
$T\perp^{\varepsilon}_{\omega}S$. Then $T\perp^{2\varepsilon}_{\omega}S+K$.
\end{proposition}
\begin{proof}
Let $\theta \in [0, 2\pi)$. There exists a sequence $\{x^{\theta}_n\}_{n\in\mathbb{N}}$ of unit vectors in $\mathscr{H}$ such that $\displaystyle\lim_{n\to \infty} |\langle Tx^{\theta}_n, x^{\theta}_n\rangle| =\omega(T)$ and $\displaystyle \lim_{n\to \infty}\langle Sx^{\theta}_n, x^{\theta}_n\rangle{\rm  Re}\{e^{-i\theta} \langle Tx^{\theta}_n, x^{\theta}_n\rangle\}\geq -\varepsilon \omega(T)$ hold.
We may assume that $\lim_{n\to \infty}\langle Sx^{\theta}_n, x^{\theta}_n\rangle>0$. We have
\begin{align*}
\lim_{n\to\infty} {\rm  Re}\{e^{-i\theta} \langle Tx^{\theta}_n, x^{\theta}_n\rangle\langle Kx^{\theta}_n, x^{\theta}_n\rangle\}&=
\lim_{n\to\infty} \langle Kx^{\theta}_n, x^{\theta}_n\rangle\lim_{n\to\infty} {\rm  Re}\{e^{-i\theta} \langle Tx^{\theta}_n, x^{\theta}_n\rangle\}\\
&\geq \lim_{n\to\infty} \langle Kx^{\theta}_n, x^{\theta}_n\rangle\frac{-\varepsilon \omega(T)}{\lim_{n\to \infty}\langle Sx^{\theta}_n, x^{\theta}_n\rangle}\\
&\geq -\varepsilon \omega(T)=-\varepsilon \omega(T)\omega(K).
\end{align*}
Hence,
\begin{align*}
 \lim_{n\to\infty}{\rm  Re}\{e^{-i\theta} \langle Tx^{\theta}_n, x^{\theta}_n\rangle\overline{(S+K)x^{\theta}_n, x^{\theta}_n\rangle}\}&=\lim_{n\to\infty} {\rm  Re}\{e^{-i\theta} \langle Tx^{\theta}_n, x^{\theta}_n\rangle\langle Sx^{\theta}_n, x^{\theta}_n\rangle\}\\
 &\quad+\lim_{n\to\infty} {\rm  Re}\{e^{-i\theta} \langle Tx^{\theta}_n, x^{\theta}_n\rangle\langle Kx^{\theta}_n, x^{\theta}_n\rangle\}\\
 &\geq -\varepsilon \omega(T)\omega(S)-\varepsilon \omega(T)\omega(K)\\
 &\geq-2\varepsilon \omega(T)\omega(S+K)\quad ({\rm as~} 0 \leq S, K\leq S+K).
 \end{align*}
Hence $T\perp^{2\varepsilon}_{\omega}S+K$.
\end{proof}

%%%%%%%%%%%%%%%%%%%%%%%%%%
%%%%%%%%%%%%%%%%%%%%%%%%%%

\section{Numerical radius derivation}
In this section, we introduce the notion of numerical radius derivation  and  provide a characterization of $\perp^{\varepsilon}_{\omega}$ by employing this notion.

Let $\theta \in [0, 2\pi)$.
The function $f:\mathbb{R}\to \mathbb{R}$ define by $f(r)=\omega^2(T+re^{i\theta}S)$ is convex.
To show this, let $r, s\in \mathbb{R}$ and $\alpha\in [0, 1]$. By the convexity of the real function $g(r)=r^2$, we have
\begin{align*}
f(\alpha r+(1-\alpha) s)&=\omega^2(T+ (\alpha r+(1-\alpha) s)e^{i\theta}S)\\
&=\omega^2(\alpha (T+re^{i\theta}S)+(1-\alpha)(T+se^{i\theta}S))\\
&\leq \Big(\alpha \omega((T+re^{i\theta}S)+(1-\alpha)\omega((T+se^{i\theta}S))\Big)^2\\
&\leq\alpha\omega^2(T+re^{i\theta}S)+(1-\alpha)\omega^2(T+se^{i\theta}S)\\
&=\alpha f(r)+(1-\alpha)f(s).
\end{align*}
Then for every $\theta \in [0, 2\pi)$ the function $D^{\theta}_{\omega}:\mathbb{B}(\mathscr{H})\times \mathbb{B}(\mathscr{H})\to \mathbb{R}$ defined by
\[D^{\theta}_{\omega}(T, S):=\lim_{r\to 0^+} \frac{\omega^2(T+re^{i\theta} S)-\omega^2(T)}{2r}\]
exists, we call it the numerical radius derivation ($\omega$-derivation).

Furthermore, the functions $f(r)=\omega^2(T+re^{i\theta}S)$ and $g(r)=2r\varepsilon \omega(T) \omega( S)$ are convex functions and so is $h(r)=\omega^2(T+re^{i\theta}S)+2r\varepsilon\omega(T) \omega( S)$.

The following theorem gives a characterization of the approximate numerical radius orthogonality for operators.
\begin{theorem}\label{31}
The relation $T\perp^{\varepsilon}_{\omega}S$ holds if and only if $\displaystyle\inf_{\theta \in [0, 2\pi)} D^{\theta}_{\omega}(T, S)
\geq -\varepsilon \omega(T) \omega(S)$.
\end{theorem}
\begin{proof}
($\Longrightarrow$) Suppose that $\theta \in [0, 2\pi)$. It follows from $T\perp^{\varepsilon}_{\omega}S$ that
$\omega^2(T+re^{i\theta} S)\geq \omega^2(T)-2r\varepsilon \omega(T) \omega( S)$ for all $r\in\mathbb{R}^+$. We have
\begin{align*}
D^{\theta}_{\omega}(T, S)&=\lim_{r\to 0^+} \frac{\omega^2(T+re^{i\theta} S)-\omega^2(T)}{2r}\\
&=\lim_{r\to 0^+} \frac{\omega^2(T+re^{i\theta} S)-\omega^2(T)+ 2r\varepsilon \omega(T) \omega(S)}{2r}+\lim_{r\to 0^+} \frac{- 2r\varepsilon \omega(T) \omega(S)}{2r}\\
&=\lim_{r\to 0^+} \frac{\omega^2(T+re^{i\theta} S)-\omega^2(T)+ 2r\varepsilon \omega(T) \omega(S)}{2r}- \varepsilon\omega(T)\omega(S).
\end{align*}
Since $\frac{\omega^2(T+re^{i\theta} S)-\omega^2(T)+ 2r\varepsilon \omega(T) \omega(S)}{2r}\geq 0$, passing to the limit, we get $D^{\theta}_{\omega}(T, S)\geq - \varepsilon \omega(T)\omega(S)$. Thus,
  $\displaystyle\inf_{\theta} D^{\theta}_{\omega}(T, S)
\geq -\varepsilon \omega(T) \omega(S)$.

($\Longleftarrow$) Let $\displaystyle\inf_{\theta \in [0, 2\pi)} D^{\theta}_{\omega}(T, S)\geq -\varepsilon \omega(T) \omega(S)$. Then for every $\theta \in [0, 2\pi)$ we have $D^{\theta}_{\omega}(T, S)\geq - \varepsilon \omega(T)\omega(S)$. Hence
\small{
 \begin{align*}
- \varepsilon \omega(T)\omega(S)\leq D^{\theta}_{\omega}(T, S)&=\lim_{r\to 0^+} \frac{\omega^2(T+re^{i\theta} S)-\omega^2(T)+ 2r\varepsilon \omega(T) \omega(S)}{2r}- \varepsilon\omega(T)\omega(S)\\
&=\frac{1}{2} \lim_{r\to 0^+} \frac{h(r)-h(0)}{r-0}- \varepsilon\omega(T)\omega(S),
 \end{align*}}
whence
$h'(0)=\lim_{r\to 0^+} \frac{h(r)-h(0)}{r-0}\geq 0$.

Then the convexity of $h$ implies that $h(r)-h(0)\geq (r-0)h'(0)\geq 0$ and so $h(r)\geq h(0)$  for every $r\geq 0$. Therefore, $\omega^2(T+re^{i\theta} S)\geq \omega^2(T)-2r\varepsilon \omega(T) \omega( S)$ for every $\theta \in [0, 2\pi)$. This entails that $T\perp^{\varepsilon}_{\omega}S$.
\end{proof}
%%%%%%%%%%%%%%%%%%%%%%%%%%%%
\begin{corollary}
The following statements are equivalent.

(i) The operator $T$ is approximate numerical radius orthogonal to $S$.

(ii) $\displaystyle\inf_{\theta \in [0, 2\pi)} D^{\theta}_{\omega}(T, S)
\geq -\varepsilon \omega(T) \omega(S)$.

(iii) For every $\theta\in[0,2\pi)$, there exists a sequence $\{x_n^{\theta}\}$  of unit vectors in $\mathscr{H}$ such that $\displaystyle\lim_{n\to \infty} |\langle Tx^{\theta}_n, x^{\theta}_n\rangle|=\omega(T)$ and $\displaystyle\lim_{n\to \infty} {\rm  Re}\{e^{-i\theta} \langle Tx^{\theta}_n, x^{\theta}_n\rangle\overline{\langle Sx^{\theta}_n, x^{\theta}_n\rangle}\}\geq -\varepsilon \omega(T) \omega(S).$
\end{corollary}
%%%%%%%%%%%%%%%%%%%%%%%%%%%%%
\begin{remark}
Lumber \cite[Theorem 11]{GL} proved that
$$\lim_{r\to 0^+} \frac{\|I+rT\|-1}{r}=\sup_{\|x\|=1} {\rm Re} \langle Tx, x\rangle$$ for any given operator $T$.

Dragomir \cite[Theorem 66]{SDR1} proved  that
$$\lim_{r\to 0^+} \frac{\omega(I+rT)-1}{r}=\sup_{\|x\|=1} {\rm Re}\, \langle Tx, x\rangle.$$
Then for every $\theta\in[0,2\pi)$, we have
\begin{align*}
D^{\theta}_{\omega}(T, I)&=\lim_{r\to 0^+} \frac{\omega^2(I+re^{i\theta} T)-1}{2r}\\
&=\lim_{r\to 0^+} \frac{\omega(I+re^{i\theta}T)-1}{r}\lim_{r\to 0^+} \frac{\omega(I+re^{i\theta}T)+1}{2}\\
&=\lim_{r\to 0^+} \frac{\omega(I+re^{i\theta}T)-1}{r}\\
&=\sup_{\|x\|=1} {\rm Re}\,\langle e^{i\theta}Tx, x\rangle.
\end{align*}
\end{remark}
%%%%%%%%%%%%%%%%%%%%%%%
\begin{remark}\label{33}
We observed that $T\perp^{\varepsilon}_{\omega}S$ if and only if $\displaystyle\inf_{\theta\in[0,2\pi)} D^{\theta}_{\omega}(T, S)\geq -\varepsilon \omega(T) \omega(S)$. In virtue of
\begin{align}\label{amy6}
\lim_{r\to 0^+}\frac{\omega(T+re^{i\theta} S)-\omega(T)}{r}&=
\lim_{r\to 0^+}\frac{\omega^2(T+re^{i\theta} S)-\omega^2(T)}{r(\omega(T+re^{i\theta} S)+\omega(T))} \nonumber\\
&=\lim_{r\to 0^+}\Big(\frac{\omega^2(T+re^{i\theta} S)-\omega^2(T)}{r}\cdot\frac{1}{\omega(T+re^{i\theta} S)+\omega(T)}\Big) \nonumber\\
&=\frac{1}{\omega(T)} \lim_{r\to 0^+}\frac{\omega^2(T+re^{i\theta} S)-\omega^2(T)}{2r} \nonumber\\
&=\frac{1}{\omega(T)}D^{\theta}_{\omega}(T, S),
\end{align}
we can say that $T\perp^{\varepsilon}_{\omega}S$ if and only if
$\displaystyle\lim_{r\to 0^+}\frac{\omega(T+re^{i\theta} S)-\omega(T)}{r}\geq -\varepsilon \omega(S)$ for every $\theta\in[0,2\pi)$.
\end{remark}
%%%%%%%%%%%%%%%%%%%%%%%%%%

If $\theta=0$, then Dragomir \cite{SDR1} proved that \[[S, T]:=\lim_{r\to 0^+} \frac{\omega^2(T+rS)-\omega^2(T)}{2r}\] gives rise to a semi-inner product-type on $\mathbb{B}(\mathscr{H})$, see also \cite{Al}.

Now, we list here some properties of  the above semi-inner product type.
\begin{lemma}
Let $\theta\in[0,2\pi)$. Then the following statements are satisfied.

(i) $[e^{i\theta}T, e^{i\theta}T]=\omega^2(T)$.

(ii) $[ie^{i\theta}T, e^{i\theta}T]=0$ and $[0, T]=[e^{i\theta}T, 0]=0$.

(iii) The following Schwarz type inequality holds
$$\left|\left[e^{i\theta}S, T\right]\right|\leq\omega(T)\omega(S)\quad T, S \in \mathbb{B}(\mathscr{H}).$$

(vi) The mapping $[e^{i\theta}S, T]$ is sub-additive in the first
variable, that is, for all $T, S, R \in \mathbb{B}(\mathscr{H})$,
 $[e^{i\theta}(S+R), T]\leq [e^{i\theta}S, T]+[e^{i\theta}R, T]$.
 \end{lemma}
\begin{proof}
(i) and (ii) are evident.
  
(iii)
  It is easy to see that
\begin{align}\label{amy7}
-\omega(S)\leq \frac{\omega(T)-r \omega(S)-\omega(T)}{r}&\leq\frac{\omega(T+re^{i\theta} S)-\omega(T)}{r} \nonumber\\
&\leq \frac{\omega(T)+r \omega(S)-\omega(T)}{r}= \omega(S).
\end{align}
It follows from \eqref{amy6} and \eqref{amy7} that
$$[e^{i\theta}S, T]=D^{\theta}_{\omega}(T, S)=\omega(T)\lim_{r\to 0^+}\frac{\omega(T+re^{i\theta} S)-\omega(T)}{r}\leq \omega(T)\omega(S)$$.

Similarly, one can show that
$[e^{i\theta}S, T]\geq -\omega(T)\omega(S)$.
Therefore,
$$|[e^{i\theta}S, T]|\leq\omega(T)\omega(S)\quad T, S \in \mathbb{B}(\mathscr{H}).$$
(iv)
Since $f(r)=\omega^2(T+re^{i\theta}S)$ is convex
\begin{align*}
\omega^2\left( \frac{2T+ re^{i\theta} (S+R)}{2}\right)&\leq\frac{1}{2}\omega^2(T+ re^{i\theta}S)+ \frac{1}{2}\omega^2(T+ re^{i\theta}R),
 \end{align*}
whence
\small
\begin{align*}
2\left(\omega^2(T+ \frac{r}{2}e^{i\theta} (S+R))-\omega^2(T)\right)&\leq\left(\omega^2(T+ re^{i\theta}S)-\omega^2(T)\right)+ \left(\omega^2(T+ re^{i\theta}R)-\omega^2(T)\right).
 \end{align*}
 Then
 \begin{align*}
 \lim_{r\to 0^+} \frac{2\Big(\omega^2(T+ \frac{r}{2}e^{i\theta} (S+R))-\omega^2(T)\Big)}{2r}
\leq  \lim_{r\to 0^+}\frac{\omega^2(T+ re^{i\theta}S)}{2r}+ \lim_{r\to 0^+}\frac{\omega^2(T+ re^{i\theta}R)}{2r}.
 \end{align*}
  Hence, $[e^{i\theta}(S+R), T]\leq [e^{i\theta}(S), T]+[e^{i\theta}(R), T]$.
   \end{proof}

 In the following proposition, we state the relation between
  $D^{\theta}_{\omega}(T, S)$ and\\
${\rm Re}\{e^{-i\theta}\langle  Tx^{\theta}_n, x^{\theta}_n\rangle\overline{\langle Sx^{\theta}_n, x^{\theta}_n\rangle}\}$ under some conditions.
%%%%%%%%%%%%%%%%%%%%%%%%%%%%%%%%
 \begin{proposition}\label{35}
 Let $\theta\in[0,2\pi)$ be fixed. If $\{x^{r}_n\} \in M^*_{\omega(T)}\cap M^*_{\omega(T+re^{i\theta}S)}$ for all $r\in \mathbb{R}^+$, then
\[[e^{i\theta}S, T]=D^{\theta}_{\omega}(T, S)=\lim_{r\to 0^+}\lim_{n\to \infty} {\rm Re}\{e^{-i\theta}\langle  Tx^{r}_n, x^{r}_n\rangle\overline{\langle Sx^{r}_n, x^{r}_n\rangle}\}.\]
\end{proposition}
\begin{proof}
We have
\begin{align*}
&\omega^2(T+re^{i\theta}S)\\
& = \lim_{n\to \infty}|\langle (T+ re^{i\theta}S)x^{r}_n, x^{r}_n\rangle|^2\\
&=\lim_{n\to \infty}\left(|\langle Tx^{r}_n, x^{r}_n \rangle|^2+2r{\rm Re}\{e^{-i\theta}\langle  Tx^{r}_n, x^{r}_n\rangle\overline{\langle Sx^{r}_n, x^{r}_n\rangle}\}+r^2|\langle Sx^{r}_n, x^{r}_n\rangle|^2 \right).
\end{align*}
Thus,
\begin{small}
\begin{align*}
&\frac{\omega^2(T+re^{i\theta}S)-\omega^2(T)}{2r}\\
&=
 \frac{\lim_{n\to \infty}\left(|\langle Tx^{r}_n, x^{r}_n \rangle|^2+2r{\rm Re}\{e^{-i\theta}\langle  Tx^{r}_n, x^{r}_n\rangle\overline{\langle Sx^{r}_n, x^{r}_n\rangle}\}+r^2|\langle Sx^{r}_n, x^{r}_n\rangle|^2\right)-\omega^2(T)}{2r} \\
&= \lim_{n\to \infty}\frac{ 2r{\rm Re}\{e^{-i\theta}\langle  Tx^{r}_n, x^{r}_n\rangle\overline{\langle Sx^{r}_n, x^{r}_n\rangle}\}+r^2|\langle Sx^{r}_n, x^{r}_n\rangle|^2 }{2r}
\end{align*}
Hence
\begin{align*}
\lim_{n\to \infty} {\rm Re}\{e^{-i\theta}\langle  Tx^{r}_n, x^{r}_n\rangle\overline{\langle S x^{r}_n, x^{r}_n\rangle}\} &\leq \frac{\omega^2(T+re^{i\theta}S)-\omega^2(T)}{2r}\\
&\leq\lim_{n\to \infty} {\rm Re}\{e^{-i\theta}\langle  Tx^{r}_n, x^{r}_n\rangle\overline{\langle Sx^{r}_n, x^{r}_n\rangle}\}  +r\omega(S).
\end{align*}
\end{small}
Letting $r\to 0^+$, we get the desired equality.
\end{proof}

\begin{remark}
Under the conditions of Proposition \ref{35}, Remark \ref{33} implies that
\small{
\begin{align*}
\lim_{r\to 0^+}\frac{\omega(T+re^{i\theta} S)-\omega(T)}{r}
&=\frac{1}{\omega(T)}D^{\theta}_{\omega}(T, S)=\frac{1}{\omega(T)}\lim_{r\to0^+}\lim_{n\to \infty}\left({\rm Re}\{e^{-i\theta}\langle  Tx^{r}_n, x^{r}_n\rangle\overline{\langle Sx^{r}_n, x^{r}_n\rangle}\}\right).
\end{align*}}
\end{remark}

%%%%%%%%%%%%%%%%%%%%%%%%%%

\end{document}